\documentclass[11pt]{article}

\usepackage{amssymb,amsmath,amsfonts,amsthm}
\usepackage{latexsym}
\usepackage{graphics}
\usepackage{indentfirst}
\usepackage{hyperref}
\usepackage{comment}

\newtheorem*{thm*}{Theorem}
\newtheorem{thm}{Theorem}
\newtheorem{lem}[thm]{Lemma}
\newtheorem{pro}[thm]{Proposition}
\newtheorem{obs}[thm]{Observation}

\newcommand{\N}{\mathbb{N}}

\begin{document}

\title{On Polynomial Representations of the DP Color Function: Theta Graphs and Their Generalizations}

\author{Charlie Halberg$^1$, Hemanshu Kaul$^2$, Andrew Liu$^1$, Jeffrey A. Mudrock$^3$, \\ Paul Shin$^1$, and Seth Thomason$^1$}

\footnotetext[1]{Department of Mathematics, College of Lake County, Grayslake, IL 60030.}

\footnotetext[2]{Department of Applied Mathematics, Illinois Institute of Technology, Chicago, IL 60616. E-mail: {\tt {kaul@iit.edu}}}

\footnotetext[3]{Department of Mathematics and Statistics, University of South Alabama, Mobile, AL 36688.  E-mail:  {\tt {mudrock@southalabama.edu}}}

\maketitle

\begin{abstract}
DP-coloring (also called correspondence coloring) is a generalization of list coloring that has been widely studied in recent years after its introduction by Dvo\v{r}\'{a}k and Postle in 2015. As the analogue of the chromatic polynomial $P(G,m)$, the DP color function of a graph $G$, denoted $P_{DP}(G,m)$, counts the minimum number of DP-colorings over all possible $m$-fold covers. It is known that, unlike the list color function  $P_{\ell}(G,m)$, for any $g \geq 3$ there exists a graph $G$ with girth $g$ such that $P_{DP}(G,m) < P(G,m)$ when $m$ is sufficiently large. Thus, two fundamental open questions regarding the DP color function are: (i) for which $G$ does there exist an $N \in \N$ such that $P_{DP}(G,m) = P(G,m)$ whenever $m \geq N$, (ii) Given a graph $G$ does there always exist an $N \in \N$ and a polynomial $p(m)$ such that $P_{DP}(G,m) = p(m)$ whenever $m \geq N$? 

In this paper we give exact formulas for the DP color function of a Theta graph based on the parity of its path lengths. This gives an explicit answer, including the formulas for the polynomials that are not the chromatic polynomial, to both the questions above for Theta graphs. We extend this result to Generalized Theta graphs by characterizing the exact parity condition that ensures the DP color function eventually equals the chromatic polynomial. To answer the second question for Generalized Theta graphs, we confirm it for the larger class of graphs with a feedback vertex set of size one.
\medskip

\noindent {\bf Keywords.}  graph coloring, list coloring, DP-coloring, correspondence coloring, chromatic polynomial, list color function, DP-color function.

\noindent \textbf{Mathematics Subject Classification.} 05C15, 05C30, 05C69

\end{abstract}

\section{Introduction}\label{intro}

In this paper all graphs are nonempty, finite, simple graphs unless otherwise noted.  Generally speaking we follow West~\cite{W01} for terminology and notation.  The set of natural numbers is $\N = \{1,2,3, \ldots \}$.  For $m \in \N$, we write $[m]$ for the set $\{1, \ldots, m \}$.  Given a set $A$, $\mathcal{P}(A)$ is the power set of $A$.  If $G$ is a graph and $S, U \subseteq V(G)$, we use $G[S]$ for the subgraph of $G$ induced by $S$, and we use $E_G(S, U)$ for the set consisting of all the edges in $E(G)$ that have one endpoint in $S$ and the other in $U$.  If $G$ and $H$ are vertex disjoint graphs, we write $G \vee H$ for the graph obtained from $G$ and $H$ by adding edges so that each vertex of $G$ is adjacent to each vertex of $H$ ($G \vee H$ is called the join of $G$ and $H$).

\subsection{List Coloring and DP-Coloring} \label{basic}

In the classical vertex coloring problem we wish to color the vertices of a graph $G$ with up to $m$ colors from $[m]$ so that adjacent vertices receive different colors, a so-called \emph{proper $m$-coloring}. The chromatic number of a graph $G$, denoted $\chi(G)$, is the smallest $m$ such that $G$ has a proper $m$-coloring.  List coloring, a well-known variation on classical vertex coloring, was introduced independently by Vizing~\cite{V76} and Erd\H{o}s, Rubin, and Taylor~\cite{ET79} in the 1970s.  For list coloring, we associate a \emph{list assignment} $L$ with a graph $G$ such that each vertex $v \in V(G)$ is assigned a list of available colors $L(v)$ (we say $L$ is a list assignment for $G$).  Then, $G$ is \emph{$L$-colorable} if there exists a proper coloring $f$ of $G$ such that $f(v) \in L(v)$ for each $v \in V(G)$ (we refer to $f$ as a \emph{proper $L$-coloring} of $G$).  A list assignment $L$ is called a \emph{$k$-assignment} for $G$ if $|L(v)|=k$ for each $v \in V(G)$.  The \emph{list chromatic number} of a graph $G$, denoted $\chi_\ell(G)$, is the smallest $k$ such that $G$ is $L$-colorable whenever $L$ is a $k$-assignment for $G$.  We say $G$ is \emph{$k$-choosable} if $k \geq \chi_\ell(G)$.  Note $\chi(G) \leq \chi_\ell(G)$, and this inequality may be strict since it is known that there are bipartite graphs with arbitrarily large list chromatic number (see~\cite{ET79}).

In 2015, Dvo\v{r}\'{a}k and Postle~\cite{DP15} introduced a generalization of list coloring called DP-coloring (they called it correspondence coloring) in order to prove that every planar graph without cycles of lengths 4 to 8 is 3-choosable. DP-coloring has been extensively studied since 2015 (see e.g.,~\cite{BH21, B16, B17, BK182, BK17, BK18, KM19, KM20, KM21, KO18, KO182, M18, MT20}).  Intuitively, DP-coloring is a variation on list coloring where each vertex in the graph still gets a list of colors, but identification of which colors are different can change from edge to edge.  Following~\cite{BK17}, we now give the formal definition. Suppose $G$ is a graph.  A \emph{cover} of $G$ is a pair $\mathcal{H} = (L,H)$ consisting of a graph $H$ and a function $L: V(G) \rightarrow \mathcal{P}(V(H))$ satisfying the following four requirements:

\vspace{5mm}

\noindent(1) the set $\{L(v) : v \in V(G)\}$ forms a partition of $V(H)$ of size $|V(G)|$;\\
(2) for every $u \in V(G)$, the graph $H[L(u)]$ is complete; \\
(3) if $E_H(L(u),L(v))$ is nonempty, then $u=v$ or $uv \in E(G)$; \\
(4) if $uv \in E(G)$, then $E_H(L(u),L(v))$ is a matching (the matching may be empty).

\vspace{5mm}

Suppose $\mathcal{H} = (L,H)$ is a cover of $G$.  We refer to the edges of $H$ connecting distinct parts of the partition $\{L(v) : v \in V(G) \}$ as \emph{cross-edges}. We say $\mathcal{H}$ is a \emph{full cover} if for each $uv \in E(G)$, the matching $E_{H}(L(u),L(v))$ is perfect. An \emph{$\mathcal{H}$-coloring} of $G$ is an independent set in $H$ of size $|V(G)|$.  It is immediately clear that an independent set $I \subseteq V(H)$ is an $\mathcal{H}$-coloring of $G$ if and only if $|I \cap L(u)|=1$ for each $u \in V(G)$.  We say $\mathcal{H}$ is \emph{$m$-fold} if $|L(u)|=m$ for each $u \in V(G)$.  The \emph{DP-chromatic number} of $G$, $\chi_{DP}(G)$, is the smallest $m \in \N$ such that $G$ has an $\mathcal{H}$-coloring whenever $\mathcal{H}$ is an $m$-fold cover of $G$.

Suppose $\mathcal{H} = (L,H)$ is an $m$-fold cover of $G$.  We say that $\mathcal{H}$ has a \emph{canonical labeling} if it is possible to name the vertices of $H$ so that $L(u) = \{ (u,j) : j \in [m] \}$ and $(u,j)(v,j) \in E(H)$ for each $j \in [m]$ whenever $uv \in E(G)$.\footnote{When $\mathcal{H}=(L,H)$ has a canonical labeling, we will always refer to the vertices of $H$ using this naming scheme.}  Clearly, when $\mathcal{H}$ has a canonical labeling, $G$ has an $\mathcal{H}$-coloring if and only if $G$ has a proper $m$-coloring.  Also, given an $m$-assignment $L$ for a graph $G$, it is easy to construct an $m$-fold cover $\mathcal{H}'$ of $G$ such that $G$ has an $\mathcal{H}'$-coloring if and only if $G$ has a proper $L$-coloring (see~\cite{BK17}).  It follows that $\chi(G) \leq \chi_\ell(G) \leq \chi_{DP}(G)$.  The second inequality may be strict since it is easy to prove that $\chi_{DP}(C_n) = 3$ whenever $n \geq 3$, but the list chromatic number of any even cycle is 2 (see~\cite{BK17} and~\cite{ET79}).

In some instances DP-coloring behaves similar to list coloring, but there are some interesting differences.  Thomassen~\cite{T94} famously proved that every planar graph has list chromatic number at most 5, and Dvo\v{r}\'{a}k and Postle~\cite{DP15} observed that the DP-chromatic number of every planar graph is at most 5.  Also, Molloy~\cite{M17} recently improved a theorem of Johansson by showing that every triangle-free graph $G$ with maximum degree $\Delta(G)$ satisfies $\chi_\ell(G) \leq (1 + o(1)) \Delta(G)/ \log(\Delta(G))$.  Bernshteyn~\cite{B17} subsequently showed that this bound also holds for the DP-chromatic number.  On the other hand, Bernshteyn~\cite{B16} showed that if the average degree of a graph $G$ is $d$, then $\chi_{DP}(G) = \Omega(d/ \log(d))$.  This is in striking contrast to the result that says $\chi_{\ell}(G) = O(\log(d))$ when $G$ is a complete bipartite graph~\cite{ET79}, and Alon~\cite{A00} showed that $\chi_\ell(G) = \Omega(\log(d))$ for any graph $G$.  It was also recently shown in~\cite{BK17} that there exist planar bipartite graphs with DP-chromatic number 4 even though the list chromatic number of any planar bipartite graph is at most 3~\cite{AT92}.  
 
\subsection{Counting Proper Colorings, List Colorings, and DP-Colorings}

In 1912 Birkhoff introduced the notion of the chromatic polynomial in hopes of using it to make progress on the four color problem.  For $m \in \N$, the \emph{chromatic polynomial} of a graph $G$, $P(G,m)$, is the number of proper $m$-colorings of $G$. It is well-known that $P(G,m)$ is a polynomial in $m$ of degree $|V(G)|$ (see~\cite{B12}).  Furthermore, both the chromatic polynomial and its generalizations are some of the central objects of study in Algebraic Combinatorics. 

The notion of chromatic polynomial was extended to list coloring in the early 1990s~\cite{AS90}. If $L$ is a list assignment for $G$, we use $P(G,L)$ to denote the number of proper $L$-colorings of $G$. The \emph{list color function} $P_\ell(G,m)$ is the minimum value of $P(G,L)$ where the minimum is taken over all possible $m$-assignments $L$ for $G$.  It is clear that $P_\ell(G,m) \leq P(G,m)$ for each $m \in \N$ since we must consider the $m$-assignment that assigns the same $m$ colors to all the vertices in $G$ when considering all possible $m$-assignments for $G$.  In general, the list color function can differ significantly from the chromatic polynomial for small values of $m$.  However, for large values of $m$,  Dong and Zhang~\cite{DZ22} (improving upon results in~\cite{D92}, \cite{T09}, and~\cite{WQ17}) showed the following in 2022.

\begin{thm} [\cite{DZ22}] \label{thm: WQ17}
For any graph $G$ with at least 4 edges, $P_{\ell}(G,m)=P(G,m)$ whenever $m \geq |E(G)|-1$.
\end{thm}

It is also known that $P_{\ell}(G,m)=P(G,m)$ for all $m \in \N$ when $G$ is a cycle or chordal (see~\cite{KN16} and~\cite{AS90}).  Moreover, if $P_{\ell}(G,m)=P(G,m)$ for all $m \in \N$, then $P_{\ell}(K_n \vee G,m)=P(K_n \vee G,m)$ for each $n, m \in \N$ (see~\cite{KM18}). See~\cite{T09} for a survey of known results and open questions on the list color function.

Two of the current authors (Kaul and Mudrock in~\cite{KM19}) introduced a DP-coloring analogue of the chromatic polynomial in hopes of gaining a better understanding of DP-coloring and using it as a tool for making progress on some open questions related to the list color function.  Suppose $\mathcal{H} = (L,H)$ is a cover of graph $G$.  Let $P_{DP}(G, \mathcal{H})$ be the number of $\mathcal{H}$-colorings of $G$.  Then, the \emph{DP color function} of $G$, $P_{DP}(G,m)$, is the minimum value of $P_{DP}(G, \mathcal{H})$ where the minimum is taken over all possible $m$-fold covers $\mathcal{H}$ of $G$.~\footnote{We take $\N$ to be the domain of the DP color function of any graph.} It is easy to show that for any graph $G$ and $m \in \N$, $P_{DP}(G, m) \leq P_\ell(G,m) \leq P(G,m)$.~\footnote{To prove this, recall that for any $m$-assignment $L$ for $G$, an $m$-fold cover $\mathcal{H}'$ of $G$ such that $G$ has an $\mathcal{H}'$-coloring if and only if $G$ has a proper $L$-coloring is constructed in~\cite{BK17}. It is easy to see from the construction in~\cite{BK17} that there is a bijection between the proper $L$-colorings of $G$ and the $\mathcal{H}'$-colorings of $G$.} Note that if $G$ is a disconnected graph with components: $H_1, H_2, \ldots, H_t$, then $P_{DP}(G, m) = \prod_{i=1}^t P_{DP}(H_i,m)$.  So, we will only consider connected graphs from this point forward unless otherwise noted.

As with list coloring and DP-coloring, the list color function and DP color function of certain graphs behave similarly.  However, for some graphs there are surprising differences.  For example, similar to the list color function,  $P_{DP}(G,m) = P(G,m)$ for every $m \in \N$  whenever $G$ is chordal or an odd cycle~\cite{KM19}.  On the other hand, we have the following result.

\begin{thm} [\cite{KM19}] \label{thm: evengirth}
If $G$ is a graph with girth that is even, then there is an $N \in \N$ such that $P_{DP}(G,m) < P(G,m)$ whenever $m \geq N$.  Furthermore, for any integer $g \geq 3$ there exists a graph $M$ with girth $g$ and an $N \in \N$ such that $P_{DP}(M,m) < P(M,m)$ whenever $m \geq N$. 
\end{thm}

This result is particularly surprising since Theorem~\ref{thm: WQ17} implies that the list color function of any graph eventually equals its chromatic polynomial. Consequently, two fundamental questions regarding the DP color function are: (i) for which $G$ does there exist an $N \in \N$ such that $P_{DP}(G,m) = P(G,m)$ whenever $m \geq N$, (ii) given a graph $G$ does there always exist an $N \in \N$ and a polynomial $p(m)$ such that $P_{DP}(G,m) = p(m)$ whenever $m \geq N$?  

Question (i) has been answered in affirmative for unicyclic graphs containing an odd cycle, chordal graphs~\cite{KM19}, and graphs with a universal vertex~\cite{MT20}. From a result in~\cite{MT20}, it is known that for every graph $G$ with $n$ vertices, $P(G,m)-P_{DP}(G,m) = O(m^{n-3})$ as $m \rightarrow \infty$; it follows that for any graph $G$ for which Question (ii) is true, the polynomial $p(m)$ will have the same three terms of highest degree as $P(G,m)$. 

\subsection{Summary of Results}

We study the two questions stated above in the context of Generalized Theta graphs and graphs with a feedback vertex set of size 1 which is a family of graphs that contains Generalized Theta graphs.  A \emph{feedback vertex set} of a graph is a subset of vertices whose removal makes the resulting induced subgraph acyclic. Clearly, a Generalized Theta graph has a feedback vertex set of size one.

A \emph{Generalized Theta graph} $\Theta(l_1, \ldots, l_k)$ consists of a pair of end vertices joined by $k$ internally disjoint paths of lengths $l_1, \ldots, l_k \in \N$. When $k=3$, $\Theta(l_1, l_2, l_3)$ is simply called a \emph{Theta graph}.  It is well-known (see~\cite{BH01}) that if $G = \Theta(l_1, \ldots, l_k)$, then for each $m \in \N$ satisfying $m \geq 2$, 
\[P(G,m) = \frac{\prod_{i=1}^{k}((m-1)^{l_i+1}+(-1)^{l_i+1}(m-1))}{(m(m-1))^{k-1}} + \frac{\prod_{i=1}^k((m-1)^{l_i}+(-1)^{l_i}(m-1))}{m^{k-1}}.\]

Generalized Theta graphs, which have been widely studied for many graph theoretic problems (see e.g.,~\cite{BF16, CM15, ET79, LB16, LB19, SJ18}), are the main subject of two classical papers on the chromatic polynomial~\cite{BH01} and~\cite{S04} which include the celebrated result that the zeros of the chromatic polynomials of the Generalized Theta graphs are dense in the whole complex plane with the possible exception of the unit disc around the origin (by including the join of Generalized Theta graphs with $K_2$ this extends to all of the complex plane). It is natural to study the DP color function of these graphs, independent of our motivating questions. 

As discussed in the previous section, we know $P_{DP}(G,m) = P(G,m)$ is not always true by Theorem~\ref{thm: evengirth}. But even if $P_{DP}(G,m) < P(G,m)$, does there always exist a polynomial $p(m)$ and $N \in \N$ such that $P_{DP}(G,m) = p(m)$ whenever $m \geq N$? This phenomenon is illustrated in~\cite{KM19} which gives exact formulas for $P_{DP}(G,m)$ when $G = \Theta(1, l_2, l_3)$ and $l_2, l_3 \geq 2$ by using a lemma (Lemma~\ref{lem: complex} below) that gives a method for computing the DP-color function of graphs that are up to a couple of edges away from being acyclic. In Section~\ref{Theta}, we extend this result to the remaining Theta graphs by a careful application of Lemma~\ref{lem: complex}.

\begin{thm} \label{thm: three theta}
	Suppose $G = \Theta(l_1, l_2, l_3)$ and $2 \leq l_1 \leq l_2 \leq l_3$. 
	
	\textit{(i)} If the parity of $l_1$ is different from both $l_2$ and $l_3$, then $P_{DP}(G,m) = P(G,m)$ for all $m\in \mathbb{N}$.
	
	\textit{(ii)} If the parity of $l_1$ is the same as $l_2$ and different from $l_3$, then  for $m \ge 2$:\\ $P_{DP}(G,m) = \dfrac{1}{m}\left[ (m-1)^{l_1+l_2+l_3} + (m-1)^{l_1} - (m-1)^{l_2+1} - (m-1)^{l_3} + (-1)^{l_3+1}(m-2) \right]$.
	
	\textit{(iii)} If the parity of $l_1$ is the same as $l_3$ and different from $l_2$, then for $m \ge 2$:\\ $P_{DP}(G,m) = \dfrac{1}{m}\left[ (m-1)^{l_1+l_2+l_3} + (m-1)^{l_1} - (m-1)^{l_3+1} - (m-1)^{l_2} + (-1)^{l_2+1}(m-2) \right]$.
	
	\textit{(iv)} If $l_1$, $l_2$ and $l_3$ all have the same parity, then for $m \geq 3$: \\ $P_{DP}(G,m) = \dfrac{1}{m}\left[ (m-1)^{l_1+l_2+l_3} - (m-1)^{l_1} - (m-1)^{l_2} -(m-1)^{l_3} + 2(-1)^{l_1+l_2+l_3} \right].$
\end{thm}

This result illustrates the complex relationship that the DP color function has with the structure of odd and even cycles in a graph. From Theorem~\ref{thm: evengirth} we know that girth being even forces $P_{DP}(G,m) < P(G,m)$ for sufficiently large $m$, but the Theorem~\ref{thm: three theta} shows girth being odd can lead to complicated behavior. Despite this, in each of the four cases we see that $P_{DP}(G,m)$ equals a polynomial for sufficiently large $m$. To handle Generalized Theta graphs, we need new ideas as the tools from~\cite{KM19} are not sufficient for dealing with graphs that require the removal of many edges to make them acyclic. 

In Section~\ref{GenTheta}, we prove the following.

\begin{thm} \label{thm: generalized}
	Suppose $G = \Theta(l_1, \ldots, l_k)$ where $k \geq 2$, $l_2 \leq \cdots \leq l_k$, and $l_2 \geq \max\{l_1,2\}$.
	
	\textit{(i)}  If there is a $j \in \{2, \ldots, k \}$ such that $l_1$ and $l_j$ have the same parity, then there is an $N \in \N$ such that $P_{DP}(G,m) < P(G,m)$ for all $m \geq N$.
	
	\textit{(ii)}  If $l_1$ and $l_j$ have different parity for each $j \in \{2, \ldots, k \}$, then there is an $N \in \N$ such that $P_{DP}(G,m) = P(G,m)$ for all $m \geq N$. 
\end{thm}

The proof of statement (ii) above takes a lot of effort through careful counting that utilizes a technique recently introduced in~\cite{MT20}.  The details are given in Section~\ref{hard}. Statement (i) is easier to prove but it does not answer the question of whether $P_{DP}(G,m)$ equals a polynomial for sufficiently large $m$. To answer that question, we study the DP color function of a class of graphs that contains all Generalized Theta graphs. Specifically, graphs with a feedback vertex set of order one. 

In Section~\ref{PartialJoin}, we show the following.

\begin{thm}\label{thm: poly}
Suppose that $G$ is a graph with a feedback vertex set of order one. Then there exists $N \in \mathbb{N}$ and a polynomial $p(m)$ such that $P_{DP}(G, m) = p(m)$ for all $m \geq N$.
\end{thm}
  
Even though we do not have an explicit formula for the polynomial $p(m)$, by the discussion at the end of Section~1.2 we know its three highest degree terms are the same as those of the corresponding chromatic polynomial. Our proof of Theorem~\ref{thm: poly} is a generalization of the proof of Lemma~\ref{lem: complex} in~\cite{KM19} as we consider a decomposition of the graph $G$ into a star $G_1$ and a spanning forest $G_0$ rather than just an induced $K_{1,2}$ and a forest. The problem then reduces to carefully counting the number of $\mathcal{H}_0$-colorings of $G_0$ that are not $\mathcal{H}$-colorings of $G$, where $\mathcal{H}_0$ is the $m$-fold cover of $G_0$ induced by a given $m$-fold cover $\mathcal{H}$ of $G$.

The graphs in Theorem~\ref{thm: poly} are a partial join of a vertex with a forest. Interestingly, it is known that for any graph $G$,  $P_{DP}(K_1 \vee G,m) = P(K_1 \vee G,m)$ for sufficiently large $m$ (see~\cite{MT20}). By Theorem~\ref{thm: generalized}, we know this conclusion (i.e., $p(m) = P(G,m)$) cannot hold for all graphs with a feedback vertex set of order one since Theorem~\ref{thm: poly} applies to Generalized Theta graphs.

\section{Theta Graphs} \label{Theta}

A Theta graph $\Theta(l_1, l_2, l_3)$ is just two edges away from being a tree. The removal of an appropriate induced path of length two in $\Theta(l_1, l_2, l_3)$ leaves a tree, allowing us to use the following Lemma and Proposition from~\cite{KM19} to prove Theorem~\ref{thm: three theta}.

\begin{lem}[\cite{KM19}] \label{lem: complex}
Suppose $G$ is a graph and $\mathcal{H} = (L, H)$ is an $m$-fold cover of $G$ with $m \geq 3$. Suppose $\alpha_1, \alpha_2, \alpha_3$ is a path of length two in $G$ and $\alpha_1\alpha_3 \not\in E(G)$. Let $e_1 = \alpha_1\alpha_2$ and $e_2 = \alpha_2\alpha_3$. Then, let $G_0 = G - \{e_1, e_2\}$, $G_1 = G - \{e_1\}$, $G_2 = G - \{e_2\}$, and $G^*$ be the graph obtained from $G$ by adding an edge between $\alpha_1$ and $\alpha_3$. Let $H' = H - (E_H (L(\alpha_1), L(\alpha_2)) \cup E_H(L(\alpha_2), L(\alpha_3)))$ so that $\mathcal{H}' = (L, H')$ is an $m$-fold cover of $G_0$.   Suppose $\mathcal{H}'$ has a canonical labeling.

Let
\begin{align*}
    A_1 &= P(G_0, m) - P(G, m),\\
    A_2 &= P(G_0, m) - P(G_2, m) + \frac{1}{m-1} P(G, m),\\
    A_3 &= P(G_0, m) - P(G_1, m) + \frac{1}{m-1} P(G, m),\\
    A_4 &= \frac{1}{m-1} \left(P(G_1, m) + P(G_2, m) + P(G^*, m) - P(G, m)\right), \text{ and}\\
    A_5 &= \frac{1}{m-1} \left( P(G_1, m) + P(G_2, m) - \frac{1}{m-2} P(G^*, m) \right).
\end{align*}
Then,
\begin{equation*}
P_{DP}(G, \mathcal{H}) \geq P(G_0, m) - \max\{A_1, A_2, A_3, A_4, A_5\}.
\end{equation*}
Moreover, there exists an $m$-fold cover of $G$, $\mathcal{H}^*$, such that
\begin{equation*}
    P_{DP}(G, \mathcal{H}^*) = P(G_0, m) - \max\{A_1, A_2, A_3, A_4, A_5\}.
\end{equation*}
\end{lem}

\begin{pro} [\cite{KM19}] \label{pro: tree}
Suppose $T$ is a tree and $\mathcal{H} = (L,H)$ is a full $m$-fold cover of $T$.  Then, $\mathcal{H}$ has a canonical labeling.
\end{pro}

We are now ready to prove Theorem~\ref{thm: three theta}. 

\begin{proof} Since $G$ contains a cycle, $P_{DP}(G,2) = P_{DP}(G,1) = 0$.  So, the result holds when $m=1,2$. Therefore, throughout this proof we suppose that $m\geq 3$. Let $u$ and $v$ be the common endpoints of the three paths of length $l_1$, $l_2$, and $l_3$ in $G$. Let $\alpha_1$ be the vertex adjacent to $u$ in the path of length $l_1$, $\alpha_2 = u$, and $\alpha_3$ be the vertex adjacent to $u$ in the path of length $l_2$. Note that $\alpha_1, \alpha_2, \alpha_3$ is a path of length 2 in $G$ and $\alpha_1\alpha_3\notin E(G)$. Suppose that $\mathcal{H} = (L,H)$ is an arbitrary $m$-fold cover of $G$. To prove the desired result, we first show that $P_{DP}(G,\mathcal{H})$ (and hence $P_{DP}(G,m)$) is at least our proposed formula for $P_{DP}(G,m)$ in each case. We may assume $\mathcal{H}$ is a full $m$-fold cover since adding edges to $H$ can only make the number of $\mathcal{H}$-colorings of $G$ smaller.

Now, we define $e_1$, $e_2$, $G_0$, $G_1$, $G_2$, $G^*$, and $\mathcal{H}'$ as they are defined in Lemma~\ref{lem: complex}. Since $G_0$ is a tree, it follows from Proposition~\ref{pro: tree} that there is a canonical labeling of $\mathcal{H}'$.

Therefore, the hypotheses of Lemma~\ref{lem: complex} are met. Using well known facts about the chromatic polynomial (see e.g.,~\cite{W01} for an overview), we make the following observation.

\begin{obs}\label{Obs: C-Poly}
The following statements hold for each $m \in \mathbb{N}$.
\begin{gather*}
P(G, m) = \frac{1}{m}\left[(m-1)^{l_1 + l_2 + l_3} + (-1)^{l_1 + l_2 + l_3}(m-1)(m-2) +  \right.\notag\\
\left. (-1)^{l_1 + l_2}(m-1)^{l_3 + 1} + (-1)^{l_1 + l_3}(m-1)^{l_2 + 1} + (-1)^{l_2 + l_3}(m-1)^{l_1 + 1}\right],\\[1.2ex]
P(G_0, m) = m(m-1)^{l_1 + l_2 + l_3 - 2},\\[1.8ex]
P(G_1, m) = (m-1)^{l_1 + l_2 + l_3 - 1} + (-1)^{l_2 + l_3}(m-1)^{l_1},\\[1.8ex]
P(G_2, m) = (m-1)^{l_1 + l_2 + l_3 - 1} + (-1)^{l_1 + l_3}(m-1)^{l_2}, and \\[1.8ex]
    P(G^*, m) = \frac{m-2}{m} \left[ (m-1)^{l_1 + l_2 + l_3 - 1} + (-1)^{l_2 + l_3}(m-1)^{l_1} + (-1)^{l_1 + l_3}(m-1)^{l_2} \right. \notag\\ 
    \left. + (-1)^{l_1 + l_2 + 1}(m-1)^{l_3 + 1} + 2 (-1)^{l_1 + l_2 + l_3} (m-1) \right].
\end{gather*}
\end{obs}

By using the  results in Observation~\ref{Obs: C-Poly}, we obtain the following expressions (using the same notation as in Lemma~\ref{lem: complex}):
\begin{align}
    A_2 - A_3 &= (-1)^{l_2+l_3}(m-1)^{l_1} + (-1)^{l_1+l_3+1}(m-1)^{l_2}, \label{eq1}\\
    A_5 - A_4 &= (-1)^{l_1+l_2}(m-1)^{l_3} + (-1)^{l_1+l_2+l_3+1}, \label{eq2}\\
    A_2 - A_1 &= (-1)^{l_2+l_3}(m-1)^{l_1}+(-1)^{l_1+l_2}(m-1)^{l_3}+(-1)^{l_1+l_2+l_3}(m-2),\label{eq3} \\
    A_3 - A_1 &= (-1)^{l_1+l_3}(m-1)^{l_2}+(-1)^{l_1+l_2}(m-1)^{l_3}+(-1)^{l_1+l_2+l_3}(m-2),\label{eq4} \\
    A_4 - A_1 &=(-1)^{l_2+l_3}(m-1)^{l_1} + (-1)^{l_1+l_3}(m-1)^{l_2} + (-1)^{l_1 + l_2+l_3}(m-2), \label{eq5}\\
    A_5 - A_2 &=(-1)^{l_1+l_3}(m-1)^{l_2} + (-1)^{l_1+l_2+l_3+1}, \text{ and}\label{eq6}\\
    A_5 - A_3 &= (-1)^{l_2+l_3}(m-1)^{l_1} + (-1)^{l_1+l_2+l_3+1}.\label{eq7}
\end{align}

Equation~(\ref{eq1}) gives $A_2 \geq A_3$ whenever $l_1 + l_3$ is odd and $A_2 \leq A_3$ otherwise. Equation~(\ref{eq2}) gives $A_5 \geq A_4$ whenever $l_1+l_2$ is even and $A_5 \leq A_4$ otherwise.  Equation~(\ref{eq3}) gives $A_2 \geq A_1$ whenever $l_1+l_2$ is even and $A_2 \leq A_1$ otherwise. Equation~(\ref{eq4}) gives $A_3 \geq A_1$ whenever $l_1+l_2$ is even and $A_3 \leq A_1$ otherwise.  Equation~(\ref{eq5}) gives $A_4 \geq A_1$ whenever $l_1+l_3$ is even and $A_4 \leq A_1$ otherwise. Equation~(\ref{eq6}) gives $A_5 \geq A_2$ whenever $l_1+l_3$ is even and $A_5 \leq A_2$ otherwise. Equation~(\ref{eq7}) gives $A_5 \geq A_3$ whenever $l_2+l_3$ is even and $A_5 \leq A_3$ otherwise.

It is easy to see why the above comparisons hold for Equations (\ref{eq1}), (\ref{eq2}), (\ref{eq6}), and (\ref{eq7}). For the remaining comparisons, some routine computations give us the required conclusions.  To see why the comparison regarding Equation~(\ref{eq3}) holds, first suppose that $l_1+l_2$ is even and it is not the case that $l_1$, $l_2$, and $l_3$ have the same parity. Then:
\begin{align*}
	A_2 - A_1
	&= (m-1)^{l_1}((-1)^{l_2 + l_3} + (m-1)^{l_3 - l_1}) + (-1)^{l_3}(m-2)\\
	&\geq (m-1)^2(-1 + (m-1)) - (m-2)\\
	&= m(m-2)^2\\
	&\geq 0.
\end{align*}

If $l_1$, $l_2$, and $l_3$ have the same parity, then:
\begin{align*}
	A_2 - A_1
	&= (m-1)^{l_1} + (m-1)^{l_3} + (-1)^{l_1 + l_2 + l_3}(m-2)\\
	&\geq 2(m-1)^2 - (m-2)\\
	&= 2m^2 - 5m + 4\\
	&\geq 0.
\end{align*}

Finally, if $l_1+l_2$ is odd, then:
\begin{align*}
	A_2 - A_1
	&= (m-1)^{l_1}((-1)^{l_2+l_3} - (m-1)^{l_3 + l_1}) + (-1)^{l_3 + 1}(m-2)\\
	&\leq (m-1)^{l_1}(1 - (m-1)) + (m-2)\\
	&= (m-2)(1 - (m-1)^{l_1})\\
	&\leq (m-2)(1-(m-1)^2)\\
	&= -m(m-2)^2\\
	&\leq 0.
\end{align*}

Now, to establish the comparison regarding Equation~(\ref{eq4}), first suppose that $l_1 + l_2$ is even and it is not the case that $l_1$, $l_2$, and $l_3$ have the same parity. Then:
\begin{align*}
	A_3 - A_1
	&= (m-1)^{l_2}((-1)^{l_1 + l_3} + (m-1)^{l_3 - l_2}) + (-1)^{l_3}(m-2)\\
	&\geq (m-1)^2(-1 + (m-1)) - (m-2)\\
	&= m(m-2)^2\\
	&\geq 0.
\end{align*}

If $l_1$, $l_2$, and $l_3$ have the same parity, then:
\begin{align*}
	A_3 - A_1
	&= (m-1)^{l_2} + (m-1)^{l_3} + (-1)^{l_1 + l_2 + l_3}(m-2)\\
	&\geq 2(m-1)^2 - (m-2)\\
	&= 2m^2 - 5m + 4\\
	&\geq 0.
\end{align*}

If $l_1 + l_2$ is odd and $l_1 + l_3$ is even, then:
\begin{align*}
	A_3 - A_1
	&= (m-1)^{l_2}(1 - (m-1)^{l_3 - l_2}) + (-1)^{1 + l_3}(m-2)\\
	&\leq (m-1)^{l_2}(1 - (m-1)) + (m-2)\\
	&= (m-2)(1 - (m-1)^{l_2})\\
	&\leq (m-2)(1 - (m-1)^2)\\
	&= -m(m-2)^2\\
	&\leq 0.
\end{align*}

If $l_1 + l_2$ is odd and $l_1 + l_3$ is odd, then:
\begin{align*}
	A_3 - A_1
	&= (m-1)^{l_2}(-1 - (m-1)^{l_3 - l_2}) + (-1)^{1+l_3}(m-2)\\
	&\leq (m-1)^{l_2}(-1 - (m-1)^0) + (m-2)\\
	&= -2(m-1)^{l_2} + (m-2)\\
	&\leq -2(m-1)^2 + (m-2)\\
	&= -2m^2 + 5m - 4\\
	&\leq 0.
\end{align*}

Finally, we establish the comparison regarding Equation~(\ref{eq5}). If $l_1 + l_3$ is even and it is not the case that $l_1$, $l_2$, and $l_3$ have the same parity, then:
\begin{align*}
	A_4 - A_1
	&= (-1)^{l_2 + l_3}(m-1)^{l_1} + (m-1)^{l_2} + (-1)^{l_2}(m-2)\\
	&\geq -(m-1)^{l_1} + (m-1)^{l_1 + 1} - (m-2)\\
	&= (m-2)((m-1)^{l_1}-1)\\
	&\geq (m-2)((m-1)^2 - 1)\\
	&= m(m-2)^{2}\\
	&\geq 0.
\end{align*}

If $l_1$, $l_2$, and $l_3$ have the same parity, then:
\begin{align*}
	A_4 - A_1
	&= (m-1)^{l_1} + (m-1)^{l_2} + (-1)^{l_1 + l_2 + l_3}(m-2)\\
	&\geq (m-1)^2 + (m-1)^3 - (m-2)\\
	&= m(m-1)^2 - (m-2)\\
	&\geq 0.
\end{align*}

If $l_1 + l_3$ is odd and $l_2 + l_3$ is even, then:
\begin{align*}
	A_4 - A_1
	&= (m-1)^{l_1} - (m-1)^{l_2} + (-1)^{1 + l_2}(m-2)\\
	&\leq (m-1)^{l_1} - (m-1)^{l_1 + 1} + (m-2)\\
	&= (m-2)(1 - (m-1)^{l_1})\\
	&\leq (m-2)(1 - (m-1)^2)\\
	&= -m(m-2)^2\\
	&\leq 0.
\end{align*}

If $l_1 + l_3$ is odd and $l_2 + l_3$ is odd, then:
\begin{align*}
	A_4 - A_1
	&= -(m-1)^{l_1} - (m-1)^{l_2} + (-1)^{1 + l_2}(m-2)\\
	&\leq -2(m-1)^2 + (m-2)\\
	&= -2m^2 + 5m - 4\\
	&\leq 0.
\end{align*} 

According to the conditions described in statements (i), (ii), (iii), and (iv) of Theorem~\ref{thm: three theta}, $\max\{A_i : i\in[5]\} = A_1$,$A_2$,$A_4$, and $A_5$, respectively. To see this, begin by considering the conditions described in statement~(i), where $l_1 + l_2$ is odd, $l_2 + l_3$ is even, and $l_1 + l_3$ is odd. From Equations (1) and (3) we have $A_1 \geq A_2 \geq A_3$, and from Equations (2) and (5) we have $A_1 \geq A_4 \geq A_5$; hence, $\max\{A_i : i\in[5]\} = A_1$. Similarly, consider the conditions described in statement~(ii), where $l_1 + l_2$ is even, $l_2 + l_3$ is odd, and $l_1 + l_3$ is odd. From Equations (1) and (4) we have $A_2 \geq A_3 \geq A_1$, and from Equations (2) and (6) we have $A_2 \geq A_5 \geq A_4$; hence, $\max\{A_i : i\in[5]\} = A_2$. Next, consider the conditions described in statement~(iii), where $l_1 + l_2$ is odd, $l_2 + l_3$ is odd, and $l_1 + l_3$ is even. From Equations (2) and (6) we have $A_4 \geq A_5 \geq A_2$, and from Equations (4) and (5) we have $A_4 \geq A_1 \geq A_3$; hence, $\max\{A_i : i \in [5]\} = A_4$. Finally, consider the conditions described in statement~(iv), where $l_1$, $l_2$, and $l_3$ have the same parity. From Equations (2) and (5) we have $A_5 \geq A_4 \geq A_1$, and from Equations (1) and (7) we have $A_5 \geq A_3 \geq A_2$; hence, $\max\{A_i : i \in [5]\} = A_5$.  Lemma~\ref{lem: complex} then implies that $P_{DP}(G,\mathcal{H})$ is at least our proposed formula for $P_{DP}(G,m)$ in each case. 

Finally, Lemma~\ref{lem: complex} also implies there exists an $m$-fold cover for which we achieve our proposed formula for $P_{DP}(G,m)$ in each case.  The result follows. 
\end{proof}

\section{Generalized Theta Graphs}\label{GenTheta}

Throughout this section, if $G = \Theta(l_1, \ldots, l_k)$ we assume $k \geq 2$, $l_2 \leq \cdots \leq l_k$, and $l_2 \geq \max\{l_1,2\}$.  Also, we will assume that the vertices that are the common endpoints in $V(G)$ are $u$ and $w$. We also let the vertices of the $i$th path be \footnote{Notice that if $l_1=1$, the first path has no internal vertices.}: $u, v_{i,1}, \ldots , v_{i,l_i-1}, w.$  Recall that it is well-known (see~\cite{BH01}) that if $G = \Theta(l_1, \ldots, l_k)$, then for each $m \in \N$, 
\[P(G,m) = \frac{\prod_{i=1}^{k}((m-1)^{l_i+1}+(-1)^{l_i+1}(m-1))}{(m(m-1))^{k-1}} + \frac{\prod_{i=1}^k((m-1)^{l_i}+(-1)^{l_i}(m-1))}{m^{k-1}}.\]

It should be noted that the above formula even holds when $k=1$ in which case $G$ is a path.  We will also make use of the following result.

\begin{lem} [\cite{KM19}] \label{lem: oneedge}
Suppose that $G$ is a graph with $e=uv \in E(G)$. If $m \geq 2$ and
$$ P(G- \{e\} , m) < \frac{m}{m-1} P(G,m),$$
then $P_{DP}(G,m) < P(G,m)$.
\end{lem}

We are now ready to focus on the proof of Theorem~\ref{thm: generalized}. We begin by proving the first half of Theorem~\ref{thm: generalized}. The proof of the second half of Theorem~\ref{thm: generalized} requires significantly more effort than the first half and will be done in Section~\ref{hard}. 

\begin{proof}
Here we prove Statement~(i). Throughout the proof of Statement~(i) suppose that $a = m-1$.  Let $e = uv_{j,1}$.  Notice that $G - \{e\}$ is a copy of $\Theta(l_1 , \ldots, l_{j-1}, l_{j+1}, \ldots, l_k)$ that shares the vertex $w$ with a path of length $l_j-1$.  Consequently, for each $m \in \N$,
\begin{align*}
&P(G- \{e\},m) \\
&= \left(  \frac{\prod_{i \in [k] - \{j\}}((a)^{l_i+1}+(-1)^{l_i+1}(a))}{((a+1)a)^{k-2}} + \frac{\prod_{i \in [k] - \{j\}}((a)^{l_i}+(-1)^{l_i}(a))}{(a+1)^{k-2}}  \right)a^{l_j - 1}.
\end{align*}

Now, it is easy to see that for any $m \geq 2$,

\begin{align*}
&\frac{m}{m-1} P(G,m) - P(G- \{e\},m) \\
&= \frac{\prod_{i=1}^{k}(a^{l_i+1}+(-1)^{l_i+1}a)}{a^k (a+1)^{k-2}} - \frac{a^{l_j - 1}\prod_{i \in [k] - \{j\}}((a)^{l_i+1}+(-1)^{l_i+1}(a))}{((a+1)a)^{k-2}} \\
&+ \frac{\prod_{i=1}^k(a^{l_i}+(-1)^{l_i}a)}{a(a+1)^{k-2}} - \frac{a^{l_j - 1}\prod_{i \in [k] - \{j\}}((a)^{l_i}+(-1)^{l_i}(a))}{(a+1)^{k-2}} \\
&= \frac{\prod_{i \in [k] - \{j\}}((a)^{l_i+1}+(-1)^{l_i+1}(a))}{a^k(a+1)^{k-2}} \left(a^{l_j+1} + (-1)^{l_j+1}a - a^{l_j+1} \right) \\
 &+ \frac{\prod_{i \in [k] - \{j\}}((a)^{l_i}+(-1)^{l_i}(a))}{a(a+1)^{k-2}} \left(a^{l_j} + (-1)^{l_j}a - a^{l_j} \right) \\
&= \frac{(-1)^{l_j}}{(a+1)^{k-2}} \left( \prod_{i \in [k] - \{j\}}((a)^{l_i}+(-1)^{l_i}(a)) - \frac{\prod_{i \in [k] - \{j\}}((a)^{l_i+1}+(-1)^{l_i+1}(a))}{a^{k-1}} \right) \\
&= \frac{(-1)^{l_j}}{(a+1)^{k-2}} \left( \prod_{i \in [k] - \{j\}}((a)^{l_i}+(-1)^{l_i}(a)) - \prod_{i \in [k] - \{j\}}((a)^{l_i}+(-1)^{l_i+1}) \right).
\end{align*}

Notice that when $l_1=1$, we know that $l_j$ is odd and $\prod_{i \in [k] - \{j\}}((a)^{l_i}+(-1)^{l_i}(a)) = 0$.  Consequently, when $l_1=1$, $\frac{m}{m-1} P(G,m) - P(G- \{e\},m) >0$ whenever $m \geq 3$, and the desired result follows by Lemma~\ref{lem: oneedge}.  So, suppose that $l_1 > 1$.  In this case the leading term of $\left( \prod_{i \in [k] - \{j\}}((a)^{l_i}+(-1)^{l_i}(a)) - \prod_{i \in [k] - \{j\}}((a)^{l_i}+(-1)^{l_i+1}) \right)$ is of the form $(-1)^{l_1} C a^{1+ \sum_{i \in [k] - \{1,j\}} l_i}$ where $C$ is some positive constant.~\footnote{We take $\sum_{i \in [k] - \{1,j\}} l_i$ to be 0 in the case that $k=2$.}  Consequently, there is an $N \in \N$ such that $\left( \prod_{i \in [k] - \{j\}}((a)^{l_i}+(-1)^{l_i}(a)) - \prod_{i \in [k] - \{j\}}((a)^{l_i}+(-1)^{l_i+1}) \right)$ is positive (resp. negative) when $l_1$ is even (resp. odd) and $m \geq N$.  Since $l_1$ and $l_j$ have the same parity, this means that $\frac{m}{m-1} P(G,m) - P(G- \{e\},m) >0$ whenever $m \geq N$, and the desired result follows by Lemma~\ref{lem: oneedge}.      
\end{proof}

\subsection{Proof of Theorem~\ref{thm: generalized} Statement~(ii)} \label{hard}

To prove Statement~(ii) of Theorem~\ref{thm: generalized}, we will build upon a method from~\cite{MT20}.  We begin by stating a simple result.

\begin{pro} [\cite{MT20}] \label{pro: obvious}
Suppose that $\mathcal{H} = (L,H)$ is an $m$-fold cover of graph $G$ and $\mathcal{H}$ has a canonical labeling.  Let $B_i = \{(v,i): v \in V(G) \}$ for each $i \in [m]$. Then, $I \subset V(H)$ satisfies: $|I \cap L(v)|=1$ for each $v \in V(G)$ and $H[I]$ is isomorphic to $G$ if and only if $I = B_j$ for some $j \in [m]$.
\end{pro}  

Throughout this subsection suppose $G = \Theta(l_1, \ldots, l_k)$ and $l_1$ has different parity than $l_j$ for each $j \in \{2, \ldots, k \}$.  This implies that $l_1 < l_2$, and any odd cycle contained in $G$ must contain the path of length $l_1$ from $u$ to $w$.

Let $l = \sum_{i=1}^k l_i$, and note that $|E(G)|=l$.  Also, let $n = l + 2 - k$, and note that $|V(G)|=n$.  Now, for each $i \in [k]$ let $e_{i}=uv_{i,1}$.~\footnote{Note that in the case that $l_1=1$, $v_{1,1}=w$.}  Then, arbitrarily name the other edges of $G$ so that $E(G) = \{e_i : i \in [l] \}$.  We want to show that $P_{DP}(G,m) = P(G,m)$ for sufficiently large $m$, or equivalently, $P_{DP}(G,m) \geq P(G,m)$ for sufficiently large $m$.  

For each $i \in [l]$ suppose that $B^{(m)}_i$ is the set of $m$-colorings of $G$ that color the endpoints of $e_i$ with the same color.~\footnote{We will just write $B_i$ when $m$ is clear from context.}  It is well-known that
$$P(G, m) =  m^n - \sum_{p=1}^l (-1)^{p-1} \left ( \sum_{1 \leq i_1 < \cdots < i_p \leq l} \left | \bigcap_{j=1}^p B_{i_j} \right| \right)$$
and $| \bigcap_{j=1}^p B_{i_j}| = m^c$ where $c$ is the number of components of the spanning subgrah of $G$ with edge set $\{e_{i_1}, \ldots, e_{i_p} \}$ (see~\cite{W32}). 

In this subsection we are interested in finding a lower bound for $P_{DP}(G,m)$.  So, whenever $\mathcal{H} = (L,H)$ is an $m$-fold cover for $G$, we will assume $\mathcal{H}$ is a full $m$-fold cover.  We will also suppose without loss of generality that $L(x) = \{(x,j): j \in [m] \}$ for each $x \in V(G)$, and $(x,j)(y,j) \in E(H)$ for each $xy \in E(G) - \{e_{2}, \ldots, e_{k}\}$ and $j \in [m]$ (this is permissible by Proposition~\ref{pro: tree} since the spanning subgraph of $G$ with edge set $E(G) - \{e_{2}, \ldots, e_{k}\}$ is a tree).  Also, for each $i \in \{2, \ldots, k \}$, let $x_{i, \mathcal{H}}$~\footnote{We will just write $x_i$ when $\mathcal{H}$ is clear from context.} be the number of edges in $E_H(L(u),L(v_{i,1}))$ that connect endpoints with differing second coordinates.  Let $\mu_\mathcal{H}$ be the smallest $i \in \{2, \ldots, k \}$ with the property that $x_{i, \mathcal{H}} > 0$.  If there is no such element in $\{2, \ldots, k\}$, let $\mu_\mathcal{H} = 0$.  Clearly, if $\mu_{\mathcal{H}}=0$, then $\mathcal{H}$ has a canonical labeling and $P_{DP}(G, \mathcal{H})=P(G,m)$.

Suppose that $\mathcal{H}= (L,H)$ is an $m$-fold cover of $G$.  Let $\mathcal{U} = \{ I \subseteq V(H) : |L(v) \cap I| = 1 \text{ for each } v \in V(G) \}$.  Clearly, $|\mathcal{U}| = m^n$.  Now, for each $i \in [l]$, suppose $e_i=y_i z_i$, and let $S_i$ be the set consisting of each $I \in \mathcal{U}$ with the property that $H[I]$ contains an edge in $E_H(L(y_i), L(z_i))$.  It is easy to show (see~\cite{MT20})
$$P_{DP}(G, \mathcal{H}) = |\mathcal{U}| - \left | \bigcup_{p=1}^l S_p \right | =  m^n - \sum_{p=1}^l (-1)^{p-1} \left ( \sum_{1 \leq i_1 < \cdots < i_p \leq l} \left | \bigcap_{j=1}^p S_{i_j} \right| \right).$$

For the next four lemmas assume that  $\mathcal{H}= (L,H)$ is an $m$-fold cover of $G$.  Let $\mu = \mu_{\mathcal{H}}$, and suppose that $\mu > 0$.  Also, suppose precisely $t$ of the integers $l_1, \ldots, l_k$ equal $l_\mu$, and assume without loss of generality that $l_{\mu} = \cdots = l_{\mu + (t-1)}$.  Using the notation above, our goal is to establish a lower bound on
$$P_{DP}(G, \mathcal{H}) - P(G,m) =  \sum_{p=1}^l (-1)^{p} \left ( \sum_{1 \leq i_1 < \cdots < i_p \leq l} \left( \left | \bigcap_{j=1}^p S_{i_j} \right| - \left| \bigcap_{j=1}^p B_{i_j} \right| \right) \right).$$

\begin{lem} \label{lem: basic}
Suppose $\{e_{i_1}, \ldots, e_{i_p} \} \subseteq E(G)$ and the spanning subgraph $G'$ of $G$ with edge set $\{e_{i_1}, \ldots, e_{i_p} \}$ has $c$ components.  Then, 
$$-m^c \leq \left | \bigcap_{j=1}^p S_{i_j} \right| - \left| \bigcap_{j=1}^p B_{i_j} \right| \leq 0.$$
\end{lem}

\begin{proof}
The first inequality follows from the fact that $\left| \bigcap_{j=1}^p S_{i_j} \right| \geq 0$ and $ \left| \bigcap_{j=1}^p B_{i_j} \right| =m^c$.  For the second inequality, we must show that $\left| \bigcap_{j=1}^p S_{i_j} \right| \leq m^c$.

Let $H' = H - \bigcup_{xy \in E(G)-E(G')} E_H(L(x),L(y))$, and suppose the components of $G'$ are $W_1, \ldots, W_c$. Note that we can construct each element $I$ of $\bigcap_{j=1}^p S_{i_j}$ in exactly one way via the following $c$ step procedure.  Furthermore, any application of the following $c$ steps will always result in an element of $\bigcap_{j=1}^p S_{i_j}$.  For each $i \in [c]$ consider the component $W_i$.  Suppose $V(W_i) = \{w_1, \ldots, w_q\}$.  Choose one element from each of $L(w_1), \ldots, L(w_q)$ so that the subgraph of $H'$ induced by the set containing these chosen elements is isomorphic to $W_i$ (this can be done in at most $m$ ways which can be seen by considering a spanning tree of $W_i$ and applying Propositions~\ref{pro: tree} and~\ref{pro: obvious}.).  Then, place these chosen elements in $I$.  Since each step can be completed in at most $m$ ways, we may conclude that $\left| \bigcap_{j=1}^p S_{i_j} \right| \leq m^c$. 
\end{proof}

\begin{lem} \label{lem: small}
Suppose $\{e_{i_1}, \ldots, e_{i_p} \} \subseteq E(G)$ and the spanning subgraph $G'$ of $G$ with edge set $\{e_{i_1}, \ldots, e_{i_p} \}$ contains no cycles or only contains cycles with fewer than $l_1 + l_\mu$ edges.  Then, 
$$\left| \bigcap_{j=1}^p S_{i_j} \right| - \left| \bigcap_{j=1}^p B_{i_j} \right| = 0.$$
\end{lem}

\begin{proof}
Let $H' = H - \bigcup_{xy \in E(G)-E(G')} E_H(L(x),L(y))$, and suppose $G'$ has $c$ components.  We must show that $\left| \bigcap_{j=1}^p S_{i_j} \right| = m^c$.  Suppose the components of $G'$ are $W_1, \ldots, W_c$.  Note that we can construct each element $I$ of $\bigcap_{j=1}^p S_{i_j}$ in exactly one way via the following $c$ step procedure.  Furthermore, any application of the following $c$ steps will always result in an element of $\bigcap_{j=1}^p S_{i_j}$.  For each $i \in [c]$ consider the component $W_i$.  Choose one element from each of $L(w_1), \ldots, L(w_q)$ so that the subgraph of $H'$ induced by the set containing these chosen elements is isomorphic to $W_i$.  Then, place these chosen elements in $I$.  We claim that this can be done in precisely $m$ ways.

If $W_i$ is a tree then Propositions~\ref{pro: tree} and~\ref{pro: obvious} imply that there are exactly $m$ ways this can be done.  If $W_i$ contains the cycles: $M_1, \ldots, M_b$, then the fact that $W_i$ contains no cycles with at least $l_1 + l_\mu$ edges implies that for any edge $xy \in \bigcup_{i=1}^b E(M_i)$, $E_H(L(x),L(y)) = \{(x,j)(y,j): j \in [m] \}$.  Furthermore, if $W_i$ does not contain any edges in $\{e_j : \mu \leq j \leq k \}$, $E_H(L(x),L(y)) = \{(x,j)(y,j): j \in [m] \}$ for every $xy \in E(W_i)$.  In the case that $W_i$ contains at least one element of $\{e_j : \mu \leq j \leq k \}$, for each $e_\gamma \in \{e_j : \mu \leq j \leq k \}$, let $r_\gamma$ be the largest element in $[l_\gamma - 1]$ so that the path $T_\gamma$ with vertices: $u, v_{\gamma,1}, \ldots, v_{\gamma,r_\gamma}$ is a path in $W_i$ (note that $v_{\gamma,r_\gamma}w \notin E(W_i)$ since $G'$ only contains cycles with fewer than $l_1 + l_\mu$ edges).  Then, rename the vertices in $L(v_{\gamma,1}), \ldots, L(v_{\gamma,r_\gamma})$ so that $E_H(L(x),L(y)) = \{(x,j)(y,j): j \in [m] \}$ for every $xy \in E(T_\gamma)$.  After this renaming is complete, we will have that $E_H(L(x),L(y)) = \{(x,j)(y,j): j \in [m] \}$ for every $xy \in E(W_i)$.  Proposition~\ref{pro: obvious} then implies that there are exactly $m$ ways to choose one element from each of $L(w_1), \ldots, L(w_q)$ so that the subgraph of $H'$ induced by the set containing these chosen elements is isomorphic to $W_i$.  

Since each step can be completed in $m$ ways, we may conclude that $\left| \bigcap_{j=1}^p S_{i_j} \right| = m^c$. 
\end{proof}

\begin{lem} \label{lem: medium}
Suppose $\{e_{i_1}, \ldots, e_{i_p} \} \subseteq E(G)$ with $p=l_1+l_\mu$ and the spanning subgraph $G'$ of $G$ with edge set $\{e_{i_1}, \ldots, e_{i_p} \}$ contains a cycle with $l_1+l_\mu$ edges.  If $e_{\mu+r} \in \{e_{i_1}, \ldots, e_{i_p} \}$ for some $0 \leq r \leq t-1$, then
$$\left | \bigcap_{j=1}^p S_{i_j} \right| - \left| \bigcap_{j=1}^p B_{i_j} \right| = -x_{\mu+r}m^{n-l_1-l_\mu}.$$
\end{lem}

\begin{proof}
First, notice that $G'$ consists of a copy of $C_{l_1+l_\mu}$ and $n-l_1-l_\mu$ isolated vertices.  Consequently, $\left| \bigcap_{j=1}^p B_{i_j} \right| = m^{n-l_1-l_\mu+1}$.  Suppose the components of $G'$ are $W_1, \ldots, W_{n-l_1-l_\mu+1}$ with $W_1 = C_{l_1+l_\mu}$.  

Now, following the idea of the proof of Lemma~\ref{lem: small}, notice that for each $i \in \{2, \ldots, k\}$, $V(W_i)$ contains a single element $y_i$, and there are $m$ ways to select an element of $L(y_i)$.  Also, we may assume that the vertices of $W_1$ in cyclic order are: $w_1, w_2, \ldots, w_{l_1+l_\mu} $ and $w_1w_{l_1+l_\mu} = e_{\mu+r}$.  Then, Proposition~\ref{pro: obvious} along with the fact that there are exactly $(m-x_{\mu+r})$ edges in $E_H(L(w_1),L(w_{l_1+l_\mu}))$ connecting vertices with the same second coordinate implies that there are exactly $(m-x_{\mu+r})$ ways to choose  one element from each of $L(w_1), \ldots, L(w_{l_1+l_\mu})$ so that the subgraph of $H$ induced by the set containing these chosen elements is isomorphic to $W_1$.  It follows that $\left| \bigcap_{j=1}^p S_{i_j} \right| = m^{n-l_1-l_\mu}(m-x_{\mu+r})$ which immediately implies the desired result.    
\end{proof}

Before stating the next Lemma, we let $s = \sum_{r=0}^{t-1} x_{\mu+r}$.

\begin{lem} \label{lem: large}
Suppose $\{e_{i_1}, \ldots, e_{i_p} \} \subseteq E(G)$ with $p \geq l_1+l_\mu+1$ and the spanning subgraph $G'$ of $G$ with edge set $\{e_{i_1}, \ldots, e_{i_p} \}$ contains a cycle with at least $l_1+l_\mu$ edges.  Then the following statements hold.
\\
\\
(i) If $p$ is odd, $\left | \bigcap_{j=1}^p S_{i_j} \right| - \left| \bigcap_{j=1}^p B_{i_j} \right| \leq 0$.
\\
\\
(ii)  If $p = l_1+l_\mu+1$, then $G'$ has $n-l_1-l_\mu$ components and contains exactly one cycle.  Furthermore, $\left | \bigcap_{j=1}^p S_{i_j} \right| - \left| \bigcap_{j=1}^p B_{i_j} \right| \geq  -2s m^{n-l_1-l_\mu-1}.$
\\
\\
(iii)  If $p > l_1 + l_\mu+1$ and $p$ is even, then $\left | \bigcap_{j=1}^p S_{i_j} \right| - \left| \bigcap_{j=1}^p B_{i_j} \right| \geq -m^{n-l_1-l_\mu-1}$    
\end{lem}

\begin{proof}
Statement~(i) immediately follows from Lemma~\ref{lem: basic}.  For Statement~(ii) suppose that $C$ is a cycle with at least $l_1+l_\mu$ edges contained in $G'$.  In the case that $C = C_{l_1+l_\mu+1}$, $E(G')=E(C)$ and it is clear that $G'$ has $n-l_1-l_\mu$ components and contains exactly one cycle.  In the case that $C = C_{l_1+l_\mu}$, $E(G')-E(C')$ contains exactly one edge, say $e$.  The fact that there is at most one edge in $G$ connecting $u$ and $w$ guarantees that the endpoints of $e$ cannot both be in $C$.  It follows that $G'$ has $n-l_1-l_\mu$ components and contains exactly one cycle.  This means that $\left| \bigcap_{j=1}^p B_{i_j} \right| = m^{n-l_1-l_\mu}$.   

Let $H' = H - \bigcup_{xy \in E(G)-E(G')} E_H(L(x),L(y))$, and suppose the components of $G'$ are $W_1, \ldots, W_{n-l_1-l_\mu+1}$, and $W_1$ is the component of $G'$ that contains $C$.  Now, we follow the idea of the procedure used in the proof Lemma~\ref{lem: small}.  For each $i \in [n-l_1-l_\mu]$ consider the component $W_i$.  Suppose $V(W_i) = \{w_1, \ldots, w_q\}$.  For $i \geq 2$, $W_i$ is a tree on at most two vertices.  So, Propositions~\ref{pro: tree} and~\ref{pro: obvious} imply that there are exactly $m$ ways to choose one element from each of $L(w_1), \ldots, L(w_q)$ so that the subgraph of $H'$ induced by the set containing these chosen elements is isomorphic to $W_i$.  Now, suppose that $i=1$. Since $W_1$ does not contain a cycle with more than $l_1+l_\mu+1$ edges, no edges from the set $\{e_j : j > \mu + t - 1 \}$ are in the cycle contained in $W_1$, and at most one edge from this set is contained in $W_1$.  If $e_{\gamma} \in \{e_j : j > \mu + t - 1 \} \cap E(W_1)$, we rename the vertices in $L(v_{\gamma,1})$ (if needed) so that $E_{H'}(L(u), L(v_{\gamma,1})) = \{(u,j)(v_{\gamma,1},j) : j \in [m] \}$.  Now, let $D_i = \{(w_j,i) : j \in [q] \}$ for each $i \in [m]$.  If $H'[D_i]$ is not isomorphic to $W_1$, then one of the elements in $D_i$ must be the endpoint of a cross edge in $H$ that connects vertices with differing second coordinates. Let $\mathcal{D}$ consist of each $D_i \in \{D_1, \ldots, D_m\}$ with the property that $H'[D_i]$ is not isomorphic to $W_1$.  Notice this means that for each $D_j \in \{D_1, \ldots, D_m\}-\mathcal{D}$, $H'[D_j]$ is isomorphic to $W_1$, and there are are least $|\{D_1, \ldots, D_m\}-\mathcal{D}|$ ways to select one element from each of $L(w_1), \ldots, L(w_q)$ so that the subgraph of $H'$ induced by the set containing these chosen elements is isomorphic to $W_1$.  Let $\mathcal{E}$ be the set of edges in $\bigcup_{r=0}^{t-1} E_{H'}(L(u),L(v_{\mu+r,1}))$ that connect vertices with differing second coordinates (note that $|\mathcal{E}|=s$).  We can construct a function $\eta: \mathcal{D} \rightarrow \mathcal{E}$ that maps each $D_i \in \mathcal{D}$ to one of the edges in $\mathcal{E}$ that has an endpoint in $D_i$.  Furthermore, if $D_i$, $D_j$, and $D_t$ are distinct elements of $\mathcal{D}$, then it is not possible for $\eta(D_i) = \eta(D_j) = \eta(D_t)$ since an edge only has two endpoints.  Consequently, $|\{D_1, \ldots, D_m\}-\mathcal{D}| \geq (m-2s)$.  So, we may conclude that $\left | \bigcap_{j=1}^p S_{i_j} \right| \geq m^{n-l_1-l_\mu}(m-2s)$, and the desired result immediately follows.

We now turn our attention to Statement~(iii).  Since $p$ is even, we know that $G'$ has at least $l_1 + l_\mu + 3$ edges.  Suppose that $C$ is a cycle with at least $l_1+l_\mu$ edges contained in $G'$.  Since $l_2 > 1$, we know that any spanning subgraph of $G$ that contains $C$ and at least $l_1+l_\mu+3$ edges has at most $(n-l_1-l_\mu - 1)$ components.  So, Statement~(iii) follows from Lemma~\ref{lem: basic}.
\end{proof}

We are now ready to establish a lower bound on $P_{DP}(G, \mathcal{H}) - P_{DP}(G,m)$ when $\mathcal{H}$ is an $m$-fold cover of $G$ with $\mu_{\mathcal{H}} > 0$.

\begin{lem} \label{lem: twist}
Suppose that $\mathcal{H}$ is an $m$-fold cover of $G$ with $\mu_{\mathcal{H}} > 0$ and $m \geq 2^{l+1}$.  Then, if $\mu = \mu_{\mathcal{H}}$, 
$$P_{DP}(G, \mathcal{H}) - P_{DP}(G,m) \geq m^{n-l_1 - l_{\mu}}- 2^{l+2}m^{n-l_1 - l_{\mu}-1}$$
\end{lem}

\begin{proof}
For this proof we will assume that same set up as in Lemmas~\ref{lem: basic} through Lemma~\ref{lem: large}.  We know by Lemma~\ref{lem: small} that
\begin{align*}
P_{DP}(G, \mathcal{H}) - P_{DP}(G,m) &=  \sum_{p=1}^l (-1)^{p} \left ( \sum_{1 \leq i_1 < \cdots < i_p \leq l} \left( \left | \bigcap_{j=1}^p S_{i_j} \right| - \left| \bigcap_{j=1}^p B_{i_j} \right| \right) \right) \\
 &= \sum_{p=l_1+l_\mu}^l (-1)^{p} \left ( \sum_{1 \leq i_1 < \cdots < i_p \leq l} \left( \left | \bigcap_{j=1}^p S_{i_j} \right| - \left| \bigcap_{j=1}^p B_{i_j} \right| \right) \right).
\end{align*}
Then, the fact that each of the edges in $\{e_i : \mu \leq i \leq \mu + t-1 \}$ appears in exactly one cycle on $(l_1+l_\mu)$ vertices in $G$ along with Lemmas~\ref{lem: small} and~\ref{lem: medium} imply
$$(-1)^{l_1+l_\mu}\sum_{1 \leq i_1 < \cdots < i_{l_1+l_\mu} \leq l} \left( \left | \bigcap_{j=1}^p S_{i_j} \right| - \left| \bigcap_{j=1}^p B_{i_j} \right| \right) = \sum_{i=0}^{t-1} x_{\mu+r}m^{n-l_1-l_\mu} = sm^{n-l_1-l_\mu}.$$
Then, Lemma~\ref{lem: small} and Statement~(ii) of Lemma~\ref{lem: large} imply 
\begin{align*}
(-1)^{l_1+l_\mu+1}\sum_{1 \leq i_1 < \cdots < i_{l_1+l_\mu+1} \leq l} \left( \left | \bigcap_{j=1}^p S_{i_j} \right| - \left| \bigcap_{j=1}^p B_{i_j} \right| \right) &\geq -\binom{l}{l_1+l_\mu+1}2s m^{n-l_1-l_\mu-1} \\
&\geq -2^{l+1}sm^{n-l_1-l_\mu-1}.
\end{align*}
Then, Lemma~\ref{lem: small} and Statement (iii) of Lemma~\ref{lem: large} imply
$$\sum_{p=l_1+l_\mu+2}^l (-1)^{p} \left ( \sum_{1 \leq i_1 < \cdots < i_p \leq l} \left( \left | \bigcap_{j=1}^p S_{i_j} \right| - \left| \bigcap_{j=1}^p B_{i_j} \right| \right) \right) \geq -2^{l}m^{n-l_1-l_\mu-1}.$$
Finally, $s > 0$, $m \geq 2^{l+1}$, and the inequalities mentioned above, yield
\begin{align*}
P_{DP}(G, \mathcal{H}) - P_{DP}(G,m) &\geq sm^{n-l_1-l_\mu} -2^{l+1}sm^{n-l_1-l_\mu-1} -2^{l}m^{n-l_1-l_\mu-1} \\
&\geq m^{n-l_1 - l_{\mu}}- 2^{l+2}m^{n-l_1 - l_{\mu}-1}.
\end{align*}
\end{proof}

We are finally ready to complete the proof of Theorem~\ref{thm: generalized} Statement~(ii).

\begin{proof}
Suppose $N \in \N$ is chosen so that $m^{n-l_1 - l_j}- 2^{l+2}m^{n-l_1 - l_j-1} \geq 0$ for each $j \in \{2, \ldots, k \}$ and $m \geq N$.  Suppose $m \geq N$ and $\mathcal{H}$ is an arbitrary $m$-fold cover of $G$.  We will show that $P_{DP}(G, \mathcal{H})-P_{DP}(G,m) \geq 0$ which will immediately imply that $P_{DP}(G,m) = P(G,m).$

First, note that if $\mu_\mathcal{H} = 0$, then $P_{DP}(G, \mathcal{H})-P_{DP}(G,m) = 0$.  So, assume that $\mu_\mathcal{H}>0$, and let $\mu = \mu_{\mathcal{H}}$.  Then, Lemma~\ref{lem: twist} along with our choice of $N$ implies that $P_{DP}(G, \mathcal{H})-P_{DP}(G,m) \geq m^{n-l_1 - l_{\mu}}- 2^{l+2}m^{n-l_1 - l_{\mu}-1} \geq 0$ whenever $m \geq N$.  
\end{proof}

\section{Graphs with a Feedback Vertex Set of Order One }\label{PartialJoin}

We begin with a simple fact that will be useful in proving Theorem~\ref{thm: poly}.


\begin{lem} \label{lem: precolor}
Suppose $G_0$ is a graph, $S \subseteq V(G_0)$, $n \geq |V(G_0)|$, and $g: S \rightarrow [n]$.  Let $T_m$ be the set of all proper $m$-colorings $f$ of $G_0$ such that $f(v) = g(v)$ for each $v \in S$. Then, $\lvert T_m\rvert = p(m)$ where $p(m)$ is a polynomial in $m$ for all $m \geq n$.
\end{lem}

\begin{proof}
First, note that $|T_n| = 0$ if and only if there are $x,y \in S$ such that $xy \in E(G_0)$ and $g(x)=g(y)$.  Consequently, when $|T_n|=0$, we have that $|T_m|=0$ whenever $m \geq n$, and the result easily follows.  So, we assume that $|T_n| > 0$.

Let $G_1$ be the graph obtained by adding edges to $G_0$ so that the vertices in $S$ form a clique. Suppose $G^*$ is the graph obtained by contracting each edge $uv\in E(G_1)$ with $u,v\in S$ and $g(u) = g(v)$ (note that since $|T_n|>0$, it is only possible to contract edges in $E(G_1)-E(G_0)$). Suppose that $s = \lvert g(S) \rvert$. For each $i\in g(S)$, let $v_i^*$ be the vertex in $G^*$ obtained after contracting the edges in $G_1[g^{-1}(i)]$. Suppose $C_m$ is the set of proper $m$-colorings of $G^*$ where $m \geq n$ and $v_i^*$ is colored with $i$ for each $i\in g(S)$. Since $B = \{v_i^*: i\in g(S)\}$ is a clique in $G^*$, it is clear that $\lvert C_m\rvert = P(G^*,m)\prod_{i=0}^{s-1}\frac{1}{m-i}$. Let $h: T_m \rightarrow C_m$ be the function defined by $h(f) = c$ where $c: V(G^*)\rightarrow [m]$ is the function given by
	\[c(v) = 
	\begin{cases}
		f(v) &\text{if $v\notin B$}\\
		i  &\text{if $v = v_i^*\in B$}.\\
	\end{cases}
	\]
	Since $h$ is a bijection, we have that \[\lvert T_m\rvert = P(G^*,m)\prod_{i=0}^{s-1}\frac{1}{m-i}.\]
	But because $G^*$ contains a clique of size $s$, we have that $P(G^*,m) = \left(\prod_{i=0}^{s-1}(m-i)\right) p(m)$ where $p(m)$ is a polynomial; therefore, $\lvert T_m \rvert = p(m)$.
\end{proof}

Now, we are ready to prove Theorem~\ref{thm: poly}. Our proof of Theorem~\ref{thm: poly} is a generalization of the proof of Lemma~\ref{lem: complex} in~\cite{KM19} as we consider a decomposition of the graph $G$ into a star $G_1$ and a spanning forest $G_0$ rather than just an induced $K_{1,2}$ and a forest. The problem then reduces to carefully counting the number of $\mathcal{H}_0$-colorings of $G_0$ that are not $\mathcal{H}$-colorings of $G$, where $\mathcal{H}_0$ is the $m$-fold cover of $G_0$ induced by a given $m$-fold cover $\mathcal{H}$ of $G$.

\begin{proof}
The result clearly holds when $G$ has no edges.  So, there is a star $G_1 = K_{1,k-1}$ with bipartition $\{\alpha_1\}, \{\alpha_2, \ldots, \alpha_k\}$ contained in $G$ such that the subgraph $G_0 = G - E(G_1)$ is a forest. 
	
We first establish a lower bound on $P_{DP}(G, m)$. Suppose $\mathcal{H} = (L,H)$ is an $m$-fold cover of $G$ with $m\geq \lvert V(G)\rvert$. We may assume that $\mathcal{H}$ is a full cover since adding edges to $H$ cannot increase the number of $\mathcal{H}$-colorings of $G$. Let $H_0 = H - \bigcup_{uv\in E(G_1)} E_H(L(u), L(v))$ so that $\mathcal{H}_0 = (L,H_0)$ is an $m$-fold cover of $G_0$. Then, by Proposition~\ref{pro: tree}, there exists a canonical labeling of $\mathcal{H}_0$; therefore, we can let $L(v) = \{(v, j) : j \in [m]\}$ so that for each $uv \in E(G_0)$, $(u, j)$ and $(v, j)$ are adjacent in $H_0$ for each $j \in [m]$.  Note that $P_{DP}(G, \mathcal{H}) = P_{DP}(G_0, \mathcal{H}_0) - |\mathcal{A}| = P(G_0, \mathcal{H}_0) - |\mathcal{A}|$ where $\mathcal{A}$ is the set of all $\mathcal{H}_0$-colorings of $G_0$ that are not $\mathcal{H}$-colorings of $G$.  So, we seek to understand $\mathcal{A}$.
	
Let $\mathcal{P} = \{P_1, \ldots, P_t\}$ be the set of all partitions of $V(G_1)$, and let $P = \{A_1, \ldots, A_k\}$ be an arbitrary element of $\mathcal{P}$. Let $W = \{ (v, i) : v \in A_i, A_i \in P \}$.  Since $P$ is a partition, notice that for any $v \in V(G_1)$, $v$ is the first coordinate of exactly one element in $W$.  We will call this element $(v,c_v)$.  Since $|P|=k \leq |V(G)|$, we can let $g_{_P} : V(G_1) \rightarrow [|V(G)|]$ be the $|V(G)|$-coloring of $G_1$ defined by $g_{_P}(v) = c_v$ for each $v \in V(G_1)$. 

For each $i \in \{2, \ldots, k\}$, let $B^{(m)}_{g_{_P(\alpha_i)}}$ be the set of all proper $m$-colorings of $G_0$ that color $\alpha_1$ with $g_{_P}(\alpha_1)$ and $\alpha_i$ with $g_{_P}(\alpha_i)$. By Lemma~\ref{lem: precolor}, we know that for any $I \subseteq \{2, \ldots, k\}$, $|\bigcap_{i \in I} B^{(m)}_{g_{_P(\alpha_i)}}|$ is a polynomial in $m$.  So, by the inclusion-exclusion principle, we know that the number of proper $m$-colorings of $G_0$ that color $\alpha_1$ with $g_{_P}(\alpha_1)$ and $\alpha_i$ with $g_{_P}(\alpha_i)$ for some $i \in \{2, \ldots, k\}$, which is given by $\lvert \bigcup_{i=2}^k B_{g_{_P}(\alpha_i)}^{(m)}\rvert$, is a polynomial in $m$ which we denote by $p_{_P}(m)$.
		
Let $H_1$ be the graph with $V(H_1) = \bigcup_{v\in V(G_1)}L(v)$ and $E(H_1) = \bigcup_{uv\in E(G_1)}E_H(L(u),L(v))$. Clearly $H_1$ can be decomposed into $m$ vertex-disjoint copies $W_1, \ldots, W_m$ of $G_1$.  Let $V(W_i) = \{(\alpha_j, c^{(i)}_j) : j \in [k]\}$. Consider any $\mathcal{H}_0$-coloring $I$ of $G_0$. The only way $I$ is not an $\mathcal{H}$-coloring of $G$ is if $H[I]$ contains at least one edge from $W_{\ell}$ for some $\ell \in [m]$. Let $\mathcal{A}_{\ell}$ be the set of every $\mathcal{H}_0$-coloring $A$ of $G_0$ where $H[A]$ contains at least one edge of $W_{\ell}$.  This means $\mathcal{A} = \bigcup_{i=1}^m A_i$.  Furthermore, since $\mathcal{A}_1, \ldots, \mathcal{A}_m$ are pairwise disjoint, $P_{DP}(G, \mathcal{H}) = P(G_0, m) - \sum_{i=1}^m |\mathcal{A}_i|$.
	
We now prove that $|\mathcal{A}_{\ell}|$ is a polynomial in $m$ for all $\ell \in [m]$.  Define an equivalence relation on $V(W_{\ell})$ such that $(\alpha_i, c_i^{(\ell)}), (\alpha_j, c_j^{(\ell)}) \in V(W_{\ell})$ are related if and only if $c_i^{(\ell)} = c_j^{(\ell)}$. Let the partition of $V(W_{\ell})$ induced by this equivalence relation be $P_{W_{\ell}}$. Clearly, there exists a unique $P_{s_{\ell}} = \{A_1, \hdots, A_k\} \in \mathcal{P}$ with the property that it is possible to let $P_{W_{\ell}} = \{B_1, \ldots, B_k\}$ so that for any $(\alpha_r, c_r^{(\ell)}) \in V(W_{\ell})$, $\alpha_r \in A_i$ if and only if $(\alpha_r, c_r^{(\ell)}) \in B_i$. Let $C : V(G_1) \rightarrow [k]$ be defined such that for each $i,j\in [k]$, $C(\alpha_j) = i$ if and only if $\alpha_j \in A_i$. Because there is a canonical labeling of $\mathcal{H}_0$, we can reinterpret $|\mathcal{A}_{\ell}|$ to be the number of proper $m$-colorings of $G_0$ that color $\alpha_1$ with $c_1^{(\ell)}$ and $\alpha_r$ with $c_r^{(\ell)}$ for some $r \in \{2, \ldots, k\}$.  This is the same as the number of proper $m$-colorings of $G_0$ that color $\alpha_1$ with $C(\alpha_1)$ and $\alpha_r$ with $C(\alpha_r)$ for some $r \in \{2, \ldots, k\}$.  So, $|\mathcal{A}_{\ell}|$ precisely matches our definition $p_{_{P_{s_{\ell}}}}(m)$ which means $|\mathcal{A}_{\ell}| = p_{_{P_{s_{\ell}}}}(m)$.
	
Using the facts we have obtained thus far, we have
	\begin{align*}
		P_{DP}(G, \mathcal{H})
		&= P(G_0, m) - \sum_{i=1}^m |\mathcal{A}_i|\\
		&= P(G_0, m) - \sum_{i=1}^m p_{_{P_{s_i}}}(m)\\
		&\geq P(G_0, m) - m \max\{p_{_P}(m) : P \in \mathcal{P}\}.
	\end{align*}
There exists an $N \geq |V(G)|$ and $P_{\mu} \in \mathcal{P}$ such that $\max\{p_{_P}(m) : P \in \mathcal{P}\}$ is equal to $p_{_{P_\mu}}(m)$ for all $m\geq N$. Therefore, we have that $P_{DP}(G,m) \geq P(G_0,m) - mp_{_{P_\mu}}(m)$ for all $m\geq N$.
	
	Suppose $m \geq N$.  Now, we show that this bound is tight by constructing an $m$-fold cover of $G$, $\mathcal{H}^* = (L^*,H^*)$ such that $P_{DP}(G, \mathcal{H}^*) = P(G_0,m) - mp_{_{P_\mu}}(m)$. Let $L^*(v) = \{(v, j): j\in \{0\}\cup[m-1]\}$ for all $v\in V(G)$, and let $V(H^*) = \bigcup_{v \in V(G)}L^*(v)$. First draw edges in $H^*$ such that $L^*(v)$ is a clique for all $v\in V(G)$. Suppose that $P_{\mu} = \{A_0, \hdots, A_{k-1}\}$. Without loss of generality, let $\alpha_1\in A_0$. For each $\alpha_{\ell}\in A_r$, draw the edges in $\{(\alpha_1, j)(\alpha_{\ell}, r + j): j\in \{0\}\cup[m-1]\}$, where we perform addition mod $m$. For each $uv\in E(G) - E(G_1)$, draw the edges in $\{(u, j)(v, j): j \in \{0\}\cup[m-1]\}$. This completes our construction, and it is easy to see from the above argument that $P_{DP}(G, \mathcal{H}^*) = P(G_0,m) - mp_{P_{\mu}}(m)$. Thus, we can conclude that $P_{DP}(G,m) = P(G_0,m) - mp_{P_{\mu}}(m)$.
\end{proof}

{\bf Acknowledgment.}  This paper is a combination of research projects conducted with undergraduate students: Charlie Halberg, Andrew Liu, Paul Shin, and Seth Thomason at the College of Lake County during the summer of 2020. The support of the College of Lake County is gratefully acknowledged.  The authors would also like to thank the anonymous referee whose comments helped the readability of this paper.

\end{document}